\documentclass[english]{amsart}
\usepackage[T1]{fontenc}
\usepackage[latin9]{inputenc}
\usepackage{mathrsfs}
\usepackage{amsbsy}
\usepackage{amsthm}
\usepackage{amssymb}

\makeatletter

\providecommand{\tabularnewline}{\\}

\numberwithin{equation}{section}
\numberwithin{figure}{section}
  \theoremstyle{definition}
  \newtheorem{defn}{\protect\definitionname}
  \theoremstyle{remark}
  \newtheorem{rem}{\protect\remarkname}
 \theoremstyle{definition}
  \newtheorem{example}{\protect\examplename}
  \theoremstyle{plain}
  \newtheorem{fact}{\protect\factname}
  \theoremstyle{plain}
  \newtheorem{lem}{\protect\lemmaname}
\theoremstyle{plain}
\newtheorem{thm}{\protect\theoremname}
\newenvironment{lyxlist}[1]
{\begin{list}{}
{\settowidth{\labelwidth}{#1}
 \setlength{\leftmargin}{\labelwidth}
 \addtolength{\leftmargin}{\labelsep}
 }}
{\end{list}}
  \theoremstyle{plain}
  \newtheorem{prop}{\protect\propositionname}
  \theoremstyle{plain}
  \newtheorem{cor}{\protect\corollaryname}

\usepackage{upgreek}
\usepackage{amssymb}

\makeatother

\usepackage{babel}
  \providecommand{\definitionname}{Definition}
  \providecommand{\examplename}{Example}
  \providecommand{\factname}{Fact}
  \providecommand{\lemmaname}{Lemma}
  \providecommand{\propositionname}{Proposition}
  \providecommand{\remarkname}{Remark}
\providecommand{\corollaryname}{Corollary}
\providecommand{\theoremname}{Theorem}

\begin{document}

\title[Poloids]{Poloids from the Points of View of Partial Transformations and Category
Theory}

\author{Dan Jonsson}

\address{Dan Jonsson, Department of Sociology, University of Gothenburg, Sweden.}

\email{dan.jonsson@gu.se}
\begin{abstract}
Monoids and groupoids are examples of poloids. On the one hand, poloids
can be regarded as one-sorted categories; on the other hand, poloids
can be represented by partial magmas of partial transformations. In
this article, poloids are considered from these two points of view.
\end{abstract}

\maketitle

\section{Introduction}

While category theory is, in a sense, a mathematical theory of mathematics,
there does also exist a mathematical (algebraic) theory of (small)
categories. The phrase ``categories are just monoidoids'' summarizes
this theory in a somewhat cryptic manner. One part of this article
is concerned with clarifying this statement, systematically developing
definitions of category-related algebraic concepts such as semigroupoids,
poloids and groupoids, and deriving results that we recognize from
category theory. While no new results are presented, the underlying
notion that (small) categories are ``just webs of monoids'' \textendash{}
or partial magmas generalizing monoids, semigroups, groups etc. \textendash{}
may deserve more systematic attention than it has received.

The other part of the article deals with the link between abstract
algebraic structures such as poloids and concrete systems of partial
transformations on some set. We obtain systems of partial transformations
that satisfy the axioms of poloids as abstract algebraic structures
by successively adding constraints on partial magmas of partial transformations;
it is also shown that every poloid is isomorphic to such a system
of partial transformations. This procedure provides an intuitive interpretation
of the poloid axioms, helping to motivate the axioms and making it
easier to discover important concepts and results. As is known, this
approach has shown its usefulness in the study of semigroups, for
example, in the work of Wagner \cite{key-9}.

At a late stage in the preparation of the manuscript, I became aware
of related work on constellations \cite{key-4,key-5,key-7}. Constellations
turned out to generalize poloids in a way that I had not considered,
yet had several points of contact with my concepts and results. I
have added an Appendix where these matters are discussed.

\section{Poloids and transformation magmas}

\subsection{(Pre)functions, (pre)transformations and magmas}
\begin{defn}
\label{def1a} A \emph{(partial) prefunction}, $\mathsf{f}:X\nRightarrow Y$
is a set $\mathfrak{X}\subseteq X$ and a rule $\overline{\mathsf{f}}$
that assigns exactly one $\overline{\mathsf{f}}\!\left(x\right)\in Y$
to each $x\in\mathcal{\mathfrak{X}}$; to simplify the notation we
may write $\overline{\mathsf{f}}\!\left(x\right)$ as $\mathsf{f}\!\left(x\right)$.
We call $\mathfrak{X}$ the \emph{domain} of $\mathsf{f}$, denoted
$\mathrm{dom}\!\left(\mathsf{f}\right)$. The \emph{image} of $\mathsf{f}$,
denoted $\mathrm{im}\!\left(\mathsf{f}\right)$, is the set $\mathsf{f}\!\left(\mathrm{dom}\!\left(\mathsf{f}\right)\right)=\left\{ \mathsf{f}\!\left(x\right)\mid x\in\mathcal{\mathfrak{X}}\right\} $. 
\end{defn}
\begin{defn}
\label{def1}A \emph{(partial) function} $f:X\nrightarrow Y$ is a
prefunction $\mathsf{f}:X\nRightarrow Y$ and a set $\mathfrak{Y}$
such that $\mathrm{im}\!\left(\mathsf{f}\right)\subseteq\mathfrak{Y}\subseteq Y$.
The \emph{domain} of $f$, denoted $\mathrm{dom}\!\left(f\right)$,
is the domain of $\mathsf{f}$, and $\mathfrak{Y}$ is called the
\emph{codomain} of $f$, denoted $\mathrm{cod}\!\left(f\right)$.
The \emph{image} of $f$, denoted $\mathrm{im}\!\left(f\right)$,
is defined to be the image of $\mathsf{f}$. 
\end{defn}

Although this terminology will not be used below, $X$ may be called
the \emph{total domain} for $\mathsf{f}:X\nRightarrow Y$ or $f:X\nrightarrow Y$,
and $Y$ may be called the \emph{total codomain} for $\mathsf{f}:X\nRightarrow Y$
or $f:X\nrightarrow Y$.

A \emph{total} \emph{prefunction} $\mathsf{f}:X\Rightarrow Y$ is
a prefunction such that $\mathrm{dom}\!\left(\mathsf{f}\right)=X$.
A \emph{non-empty} \emph{prefunction} $\mathsf{f}$ is a prefunction
such that $\mathrm{dom}\!\left(\mathsf{f}\right)=\emptyset$. The
\emph{restriction} of $\mathsf{f}:X\nRightarrow Y$ to $X'\subset X$
is the prefunction $\mathfrak{\mathsf{f}}\raise-.2ex\hbox{\ensuremath{|}}_{X'}:X'\nRightarrow Y$
such that $\mathrm{dom}\!\left(\mathsf{\mathsf{f}}\raise-.2ex\hbox{\ensuremath{|}}_{X'}\right)=\mathrm{dom}\!\left(\mathsf{f}\right)\cap X'$
and $\mathsf{f}\raise-.2ex\hbox{\ensuremath{|}}_{X'}\!\left(x\right)=\mathsf{f}\!\left(x\right)$
for all $x\in\mathrm{dom}\!\left(\mathsf{f}\raise-.2ex\hbox{\ensuremath{|}}_{X'}\right)$.
A\emph{ pretransformation} on $X$ is a prefunction $\mathsf{f}:X\nRightarrow X$;
a \emph{total} pretransformation on $X$ is a total prefunction $\mathsf{f}:X\Rightarrow X$.
An \emph{identity pretransformation} $\mathsf{Id}_{S}$ is a pretransformation
on $X\supseteq S$ such that $\mathrm{dom}\!\left(\mathsf{Id}_{S}\right)=S$
and $\mathsf{Id}_{S}\!\left(x\right)=x$ for all $x\in\mathrm{dom}\!\left(\mathsf{Id}_{S}\right)$. 

Similarly, a \emph{total} \emph{function} $f:X\rightarrow Y$ is a
function such that $\mathrm{dom}\!\left(f\right)=X$ and $\mathrm{cod}\!\left(f\right)=Y$.
A \emph{non-empty function} $f$ is a function such that $\mathrm{dom}\!\left(f\right)~\neq~\emptyset$.
The \emph{restriction} of $f:X\nrightarrow Y$ to $X'\subset X$ is
the  function $f\raise-.2ex\hbox{\ensuremath{|}}_{X'}:X'\nrightarrow Y$
such that $\mathrm{dom}\!\left(f\raise-.2ex\hbox{\ensuremath{|}}_{X'}\right)=\mathrm{dom}\!\left(f\right)\cap X'$,
$\mathrm{cod}\!\left(f\raise-.2ex\hbox{\ensuremath{|}}_{X'}\right)=\mathrm{cod}\!\left(f\right)$
and $f\raise-.2ex\hbox{\ensuremath{|}}_{X'}\!\left(x\right)=f\!\left(x\right)$
for all $x\in\mathrm{dom}\!\left(f\raise-.2ex\hbox{\ensuremath{|}}_{X'}\right)$.
A\emph{ transformation} on $X$ is a function $f:X\nrightarrow X$;
a \emph{total} transformation on $X$ is a total function $f:X\rightarrow X$.
An \emph{identity transformation} $I\!d_{S}$ is a transformation
on $X\supseteq S$ such that $\mathrm{dom}\!\left(I\!d_{S}\right)=\mathrm{cod}\!\left(I\!d_{S}\right)=S$
and $I\!d_{S}\!\left(x\right)=x$ for all $x\in\mathrm{dom}\!\left(I\!d_{S}\right)$. 

Given a pretransformation $\mathsf{f}$ on $X$, $\mathsf{f}\!\left(x\right)$
denotes some $x\in X$ if and only if $x\in\mathrm{dom}\!\left(\mathsf{f}\right)$;
$\mathsf{f}\!\left(\mathsf{f}\!\left(x\right)\right)$ denotes some
$x\in X$ if and only if $x,\mathsf{\mathsf{f}}\!\left(x\right)\in\mathrm{dom}\!\left(\mathsf{f}\right)$;
etc. We describe such situations by saying that $\mathsf{f}\!\left(x\right)$,
$\mathsf{f}\!\left(\mathsf{f}\!\left(x\right)\right)$, etc. are \emph{defined}.
Similarly, given a transformation $f$ on $X$, $f\!\left(x\right)$,
$f\!\left(f\!\left(x\right)\right)$, etc. are said to be defined
if the corresponding pretransformations $\mathsf{f}\!\left(x\right)$,
$\mathsf{f}\!\left(\mathsf{f}\!\left(x\right)\right)$, etc. are defined\emph{.}
\begin{defn}
\label{def2}A \emph{(partial) binary operation} on a set $X$ is
a non-empty prefunction
\[
\uppi:X\times X\nRightarrow X,\qquad\left(x,y\right)\mapsto xy.
\]
\end{defn}
A \emph{total} binary operation on $X$ is a total prefunction $\uppi:X\times X\Rightarrow X$.
A \emph{(partial) magma} $P$ is a non-empty set $\left|P\right|$
equipped with a binary operation on $\left|P\right|$; a \emph{total}
magma $P$ is a non-empty set $\left|P\right|$ equipped with a total
binary operation on $\left|P\right|$. A \emph{submagma} $P'$ of
a magma $P$ is a set $\left|P'\right|\subseteq\left|P\right|$ such
that if $x,y\in\left|P'\right|$ then $xy\in\left|P'\right|$, with
the restriction of $\uppi$ to $\left|P'\right|\times\left|P'\right|$
as a binary operation. (By an abuse of notation, $P$ will also denote
the set $\left|P\right|$ henceforth.)

The notion of being defined for expressions involving a pretransformation
can be extended in a natural way to expressions involving a binary
operation. We say that $xy$ is defined if and only if $\left(x,y\right)\in\mathrm{dom}\!\left(\uppi\right)$;
that $\left(xy\right)\!z$ is defined if and only if $\left(x,y\right),\left(xy,z\right)\in\mathrm{dom}\!\left(\uppi\right)$;
that $z\!\left(xy\right)$ is defined if and only if $\left(x,y\right),\left(z,xy\right)\in\mathrm{dom}\!\left(\uppi\right)$;
and so on. Thus, if $\left(xy\right)\!z$ or $z\!\left(xy\right)$
is defined then $xy$ is defined.
\begin{rem}
{\small{}To avoid tedious repetition of the word ``partial'', we
speak about (pre)functions and magmas as opposed to total (pre)functions
and total magmas rather than partial (pre)functions and partial magmas
as opposed to (pre)functions and magmas. Note that a binary operation
$\uppi:P\times P\nRightarrow P$ can always be regarded as a total
binary operation $\uppi^{0}:P^{0}\times P^{0}\Rightarrow P^{0}$,
where $P^{0}=P\cup\left\{ 0\right\} $ and $0x=x0=0$ for each $x\in P$,
considering $xy$ to be defined if and only if $xy\neq0$. If we let
$P^{0}$ represent $P$ in this way, it becomes a theorem that if
$\left(xy\right)\!z$ or $z\!\left(xy\right)$ is defined then $xy$
is defined.}{\small \par}
\end{rem}

\subsection{Semigroupoids, poloids and groupoids }

We say that $x$ \emph{precedes} $y$, denoted $x\prec y$, if and
only if $xy$ or $x\!\left(yz\right)$ or $\left(zx\right)\!y$ is
defined, and we write $x\prec y\prec z$ if and only if $x\prec y$
and $y\prec z$, meaning that $xy$ and $yz$ are defined or $x\!\left(yz\right)$
is defined or $\left(xy\right)\!z$ is defined.
\begin{defn}
\label{def3}A \emph{semigroupoid} is a magma $P$ such that for any
\linebreak{}
 $x\prec y\prec z\in P$, $\left(xy\right)\!z$ and $x\!\left(yz\right)$
are defined and $\left(xy\right)\!z=x\!\left(yz\right)$.
\end{defn}
A\emph{ unit} in a magma $P$ is any $e\in P$ such that $ex=x$ for
all $x$ such that $ex$ is defined and $xe=x$ for all $x$ such
that $xe$ is defined. 
\begin{defn}
\label{def4}A \emph{poloid} is a semigroupoid $P$ such that for
any $x\in P$ there are units $\epsilon_{\!x},\varepsilon_{\!x}\in P$
such that $\epsilon_{x}x$ and $x\varepsilon_{\!x}$ are defined.
\end{defn}
For any $x\in P$, we have $\epsilon_{x}x=x=x\varepsilon_{x}$ since
$\epsilon_{x}$ and $\varepsilon_{x}$ are units; we may call $\epsilon_{x}$
an \emph{effective left unit} for $x$ and $\varepsilon_{x}$ an \emph{effective
right unit} for $x$.
\begin{defn}
\label{def5}A \emph{groupoid} is a poloid $P$ such that for every
$x\in P$ there is a unique $x^{-1}\in P$ such that $xx^{-1}$ and
$x^{-1}x$ are defined and units.
\end{defn}
\begin{rem}
{\small{}Recall that groups, monoids and semigroups are total magmas
with additional properties. Each kind of total magma can be generalized
to a (partial) magma with similar properties, sometimes named by adding
the ending ``-oid'', as in group/groupoid and semigroup/semigroupoid,
so that the process of generalizing to a not necessarily total magma
has become known as ``oidification''. (See the table below.) However,
the terminology is not consistent \textendash{} for example, a monoid
is not a (partial) magma. I prefer ``poloid'' to the rather clumsy
and confusing term ``monoidoid'', which suggests some kind of ``double
oidification''. An important concept should have a short name, and
the idea behind the current terminology is that a monoid has a single
unit, whereas a poloid may have more than one unit. \bigskip{}
}{\small \par}

{\small{}}%
\begin{tabular}{ll}
\hline 
{\small{}total magma (magma, groupoid)} & {\small{}magma (partial magma, halfgroupoid)}\tabularnewline
{\small{}semigroup} & {\small{}semigroupoid}\tabularnewline
{\small{}monoid} & {\small{}poloid (monoidoid)}\tabularnewline
{\small{}group} & {\small{}groupoid}\tabularnewline
\hline 
\end{tabular}{\small \par}

\medskip{}
\end{rem}
{\small{}It should be kept in mind that semigroups, monoids and groups
can be generalized to other (partial) magmas than semigroupoids, poloids
and groupoids, respectively. For example, if we do not require that
if $x\prec y\prec z$ then $x\!\left(yz\right)$ and $\left(xy\right)\!z$
are defined and equal but only that if $x\!\left(yz\right)$ or $\left(xy\right)\!z$
is defined then $x\!\left(yz\right)$ and $\left(xy\right)\!z$ are
defined and equal, we obtain a semigroup generalized to a certain
(partial) magma but this is not a semigroupoid as defined here. The
specific definitions given in this section are suggested by category
theory.}{\small \par}

\subsection{(Pre)transformation magmas}

Recall that the \emph{full transformation monoid} $\overline{\mathcal{F}}{}_{\!X}$
on a non-empty set $X$ is the set $\overline{\mathcal{F}}{}_{\!X}$
of all total functions $f:X\rightarrow X$, equipped with the total
binary operation
\[
\circ:\overline{\mathcal{F}}{}_{\!X}\times\overline{\mathcal{F}}{}_{\!X}\Rightarrow\overline{\mathcal{F}}{}_{\!X},\qquad\left(f,g\right)\mapsto f\circ g,
\]
where $\left(f\circ g\right)\!\left(x\right)=f\!\left(g\!\left(x\right)\right)$
for all $x\in X$. More generally, a \emph{transformation semigroup}
$\mathcal{F}{}_{\!X}$ is a set of total functions $f:X\rightarrow X$
with $\circ$ as binary operation and such that $f,g\in\mathcal{F}{}_{\!X}$
implies $\mathsf{f}\circ\mathsf{g}\in\mathcal{F}{}_{\!X}$, and a
\emph{transformation monoid} $\mathcal{M}_{\!X}$ is a transformation
semigroup such that $I\!d{}_{\!X}\in\mathcal{M}{}_{\!X}$.
\begin{example}
\label{exa1}Set $X=\left\{ 1,2\right\} $, let $e:X\rightarrow X$
be defined by $e\!\left(1\right)=e\!\left(2\right)=1$ and let $M_{X}$
be the magma with $\left\{ e\right\} $ as underlying set and function
composition $\circ$ as binary operation. Then $M_{X}$ is a (trivial)
\emph{monoid of transformations}, but it is not a \emph{transformation
monoid}.
\end{example}
When we generalize from total functions $X\rightarrow X$ to functions
$X\nrightarrow X$ or prefunctions $X\nRightarrow X$, $\mathcal{F}{}_{\!X}$
is generalized from a transformation semigroup to a transformation
magma $\mathscr{F}_{\!X}$ or a pretransformation magma $\mathscr{R}_{\!X}$.
\begin{defn}
\label{def6a}Let $X$ be a non-empty set. A\emph{ pretransformation
magma} $\mathscr{R}_{\!X}$ on $X$ is a set $\mathscr{R}_{X}$ of
non-empty pretransformations $\mathsf{f}:X\nRightarrow X$, equipped
with the binary operation
\[
\circ:\mathscr{R}_{\!X}\times\mathscr{R}_{\!X}\nRightarrow\mathscr{R}_{\!X},\qquad\left(\mathsf{f},\mathsf{g}\right)\mapsto\mathsf{f}\circ\mathsf{g},
\]
where $\mathrm{dom}\!\left(\circ\right)=\left\{ \left(\mathsf{f},\mathsf{g}\right)\mid\mathrm{dom}\!\left(\mathsf{f}\right)\supseteq\mathsf{im}\!\left(\mathsf{g}\right)\right\} $
and $\mathsf{f}\circ\mathsf{g}$ if defined is given by $\mathsf{\mathrm{dom}}\!\left(\mathsf{f}\circ\mathsf{g}\right)=\mathrm{dom}\!\left(\mathsf{g}\right)$
and $\left(\mathsf{f}\circ\mathsf{g}\right)\!\left(x\right)=\mathsf{f}\!\left(\mathsf{g}\!\left(x\right)\right)$
for all $x\!\in\!\mathrm{dom}\!\left(\mathsf{f}\circ\mathsf{g}\right)$. 
\end{defn}
The \emph{full} pretransformation magma on $X$, denoted $\overline{\mathscr{R}}_{\!X}$,
is the pretransformation magma whose underlying set is the set of
all non-empty pretransformations of the form $\mathsf{f}:X\nRightarrow X$.
\begin{defn}
\label{def6}Let $X$ be a non-empty set. A\emph{ transformation magma}
$\mathscr{F}_{\!X}$ on $X$ is a set $\mathscr{F}_{X}$ of non-empty
transformations $f:X\nrightarrow X$, equipped with the binary operation
\[
\circ:\mathscr{F}_{\!X}\times\mathscr{F}_{\!X}\nRightarrow\mathscr{F}_{\!X},\qquad\left(f,g\right)\mapsto f\circ g,
\]
where $\mathrm{dom}\!\left(\circ\right)=\left\{ \left(f,g\right)\mid\mathrm{dom}\!\left(f\right)\supseteq\mathsf{im}\!\left(g\right)\right\} $
and $f\circ g$ if defined is given by $\mathsf{\mathrm{dom}}\!\left(f\circ g\right)=\mathrm{dom}\!\left(g\right)$,
$\mathrm{cod}\!\left(f\circ g\right)=\mathrm{cod}\!\left(f\right)$
and $\left(f\circ g\right)\!\left(x\right)=f\!\left(g\!\left(x\right)\right)$
for all $x\!\in\!\mathrm{dom}\!\left(f\circ g\right)$. 
\end{defn}
The \emph{full} transformation magma on $X$, denoted $\mathscr{\overline{F}}_{\!X}$,
is the transformation magma whose underlying set is the set of all
non-empty  transformations $\mathsf{f}:X\nrightarrow X$.

A (pre)transformation magma is clearly a magma as described in Definition
\ref{def2}.

The plan in this section, derived from the view that categories are
``webs of monoids'', is to construct transformation magmas that
relate to poloids in the same way that transformation monoids relate
to monoids. As a monoid is an \emph{associative} magma with a \emph{unit},
we look for appropriate generalizations of these two notions.
\begin{fact}
\label{f1}Let $\mathsf{f},\mathsf{g},\mathsf{h}$ be elements of
a pretransformation magma.\emph{ }If $\left(\mathsf{f}\circ\mathsf{g}\right)\circ\mathsf{h}$
and $\mathsf{f}\circ\left(\mathsf{g}\circ\mathsf{h}\right)$ are defined
then $\left(\mathsf{f}\circ\mathsf{g}\right)\circ\mathsf{h}=\mathsf{f}\circ\left(\mathsf{g}\circ\mathsf{h}\right)$.
\end{fact}
\begin{proof}
We have 
\begin{gather*}
\mathrm{dom}\!\left(\left(\mathsf{f}\circ\mathsf{g}\right)\circ\mathsf{h}\right)=\mathrm{dom}\!\left(\mathsf{h}\right)=\mathrm{dom}\!\left(\mathsf{g}\circ\mathsf{h}\right)=\mathrm{dom}\!\left(\mathsf{f}\circ\left(\mathsf{g}\circ\mathsf{h}\right)\right),
\end{gather*}
and 
\[
\left(\left(\mathsf{f}\circ\mathsf{g}\right)\circ\mathsf{h}\right)\!\left(x\right)=\left(\mathsf{f}\circ\mathsf{g}\right)\left(\mathsf{h}\!\left(x\right)\right)=\mathsf{f}\!\left(\mathsf{g}\!\left(\mathsf{h}\!\left(x\right)\right)\right)=\mathsf{f}\!\left(\left(\mathsf{g}\circ\mathsf{h}\right)\!\left(x\right)\right)=\left(\mathsf{f}\circ\left(\mathsf{g}\circ\mathsf{h}\right)\right)\!\left(x\right)
\]
for all $x\in\mathrm{dom}\!\left(\left(\mathsf{f}\circ\mathsf{g}\right)\circ\mathsf{h}\right)=\mathrm{dom}\!\left(\mathsf{f}\circ\left(\mathsf{g}\circ\mathsf{h}\right)\right)$.
\end{proof}
\begin{lem}
\label{lem1}Let $\mathsf{f,g}$ be elements of a pretransformation
magma. If $\mathsf{f}\circ\mathsf{g}$ is defined then $\mathsf{im}\!\left(\mathsf{f}\right)\supseteq\mathsf{im}\!\left(\mathsf{f}\circ\mathsf{g}\right)$.
\end{lem}
\begin{proof}
Since $\mathrm{dom}\!\left(\mathsf{f}\right)\supseteq\mathsf{im}\!\left(\mathsf{g}\right)$
by definition, we have $\mathsf{im}\!\left(\mathsf{f}\right)=\mathsf{f}\!\left(\mathrm{dom}\!\left(\mathsf{f}\right)\right)\supseteq\mathsf{f}\!\left(\mathsf{im}\!\left(\mathsf{g}\right)\right)=\mathsf{f}\!\left(\mathsf{g}\!\left(\mathrm{dom}\!\left(\mathsf{g}\right)\right)\right)=\left(\mathsf{f}\circ\mathsf{g}\right)\left(\mathrm{dom}\!\left(\mathsf{f}\circ\mathsf{g}\right)\right)=\mathsf{im}\!\left(\mathsf{f}\circ\mathsf{g}\right)$.
\end{proof}
\begin{fact}
\label{f2}Let $\mathsf{f},\mathsf{g},\mathsf{h}$ be elements of
a pretransformation magma. If $\mathsf{f}\circ\mathsf{g}$ and $\mathsf{g}\circ\mathsf{h}$
are defined then $\left(\mathsf{f}\circ\mathsf{g}\right)\circ\mathsf{h}$
and $\mathsf{f}\circ\left(\mathsf{g}\circ\mathsf{h}\right)$ are defined.
\end{fact}
\begin{proof}
We have $\mathrm{dom}\!\left(\mathsf{f}\circ\mathsf{g}\right)=\mathrm{dom}\!\left(\mathsf{g}\right)\supseteq\mathsf{im}\!\left(\mathsf{h}\right)$,
so $\left(\mathsf{f}\circ\mathsf{g}\right)\circ\mathsf{h}$ is defined.
Also, $\mathrm{dom}\!\left(\mathsf{f}\right)\supseteq\mathsf{im}\!\left(\mathsf{g}\right)$
and by Lemma \ref{lem1} $\mathsf{im}\!\left(\mathsf{g}\right)\supseteq\mathsf{im}\!\left(\mathsf{g}\circ\mathsf{h}\right)$,
so $\mathsf{f}\circ\left(\mathsf{g}\circ\mathsf{h}\right)$ is defined.
\end{proof}
\begin{fact}
\label{f3}Let $\mathsf{f},\mathsf{g},\mathsf{h}$ be elements of
a pretransformation magma{\small{}.} If $\left(\mathsf{f}\circ\mathsf{g}\right)\circ\mathsf{h}$
is defined then $\mathsf{f}\circ\left(\mathsf{g}\circ\mathsf{h}\right)$
is defined.
\end{fact}
\begin{proof}
If $\left(\mathsf{f}\circ\mathsf{g}\right)\circ\mathsf{h}$ is defined
so that $\mathsf{f}\circ\mathsf{g}$ is defined then $\mathrm{dom}\!\left(\mathsf{g}\right)=\mathrm{dom}\!\left(\mathsf{f}\circ\mathsf{g}\right)\supseteq\mathsf{im}\!\left(\mathsf{h}\right)$.
Thus, $\mathsf{g}\circ\mathsf{h}$ is defined so Fact \ref{f2} implies
that $\mathsf{f}\circ\left(\mathsf{g}\circ\mathsf{h}\right)$ is defined.
\end{proof}
The implication in the opposite direction does not hold.
\begin{example}
\label{exa2}Let $\mathsf{f,g,h}$ be pretransformations on $\left\{ 1,2\right\} $;
specifically, $\mathsf{f}=\mathsf{h}=\mathsf{Id}_{\left\{ 1\right\} }$
and $\mathsf{g}=\mathsf{Id}_{\left\{ 1,2\right\} }$. Then, $\mathrm{dom}\!\left(\mathsf{g}\right)\supseteq\mathrm{im}\!\left(\mathsf{h}\right)$
and $\mathrm{im}\!\left(\mathsf{g}\circ\mathsf{h}\right)=\left\{ 1\right\} $,
so $\mathrm{dom}\!\left(\mathsf{f}\right)\supseteq\mathrm{im}\!\left(\mathsf{\mathsf{g}\circ\mathsf{h}}\right)$.
Hence, $\mathsf{f}\circ\left(\mathsf{g}\circ\mathsf{h}\right)$ is
defined, but we do not have $\mathrm{dom}\!\left(\mathsf{f}\right)\supseteq\mathrm{im}\!\left(\mathsf{g}\right)$,
so $\mathsf{f}\circ\mathsf{g}$ is not defined and hence $\left(\mathsf{f}\circ\mathsf{g}\right)\circ\mathsf{h}$
is not defined.
\end{example}
So, somewhat surprisingly, pretransformation magmas do not have a
two-sided notion of associativeness. We need the notion of a transformation
magma and an additional assumption to derive the complement of Fact
\ref{f3}.
\begin{defn}
\label{def7}A\emph{ transformation semigroupoid} $\mathscr{S}_{\!X}$
on $X$ is a transformation magma $\mathscr{F}_{\!X}$ such that if
$\mathrm{dom}\!\left(f\right)\supseteq\mathsf{im}\!\left(g\right)$
for some $f,g\in\mathscr{F}_{\!X}$ then $\mathrm{dom}\!\left(f\right)=\mathsf{\mathrm{cod}}\!\left(g\right)$. 
\end{defn}
Of course, if $\mathrm{dom}\!\left(f\right)=\mathrm{cod}\!\left(g\right)$
then $\mathrm{dom}\!\left(f\right)\supseteq\mathsf{im}\!\left(g\right)$.
Thus, in a transformation semigroupoid $f\circ g$ is defined if and
only if $\mathrm{dom}\!\left(f\right)=\mathsf{\mathrm{cod}}\!\left(g\right)$.

If $\left(f\circ g\right)\circ h$ and $f\circ\left(g\circ h\right)$
are defined then $\mathrm{cod}\!\left(\left(f\circ g\right)\circ h\right)=\mathrm{cod}\!\left(f\circ g\right)=\mathrm{cod}\!\left(f\right)=\mathrm{cod}\!\left(f\circ\left(g\circ h\right)\right)$,
so Fact 1 holds for transformation magmas as well. It is also clear
that the proofs of Facts 2 and 3 apply to transformation magmas as
well. Thus, we can use Facts 1\textendash 3 also when dealing with
transformation magmas. On the other hand, Example 2 applies to transformation
magmas as well, but not to transformation semigroupoids. 
\begin{fact}
\label{f4}Let $f,g,h$ be elements of a transformation semigroupoid.
If $f\circ\left(g\circ h\right)$ is defined then $\left(f\circ g\right)\circ h$
is defined.
\end{fact}
\begin{proof}
If $f\circ\left(g\circ h\right)$ is defined then $\mathrm{dom}\!\left(f\right)=\mathrm{cod}\!\left(g\circ h\right)=\mathsf{\mathrm{cod}}\!\left(g\right)$.
Thus, $f\circ g$ is defined, and as $g\circ h$ is defined as well
Fact \ref{f2} implies that {\small{}$\left(f\circ g\right)\circ h$
}is defined.
\end{proof}
\begin{thm}
\label{the1}A transformation semigroupoid is a semigroupoid.
\end{thm}
\begin{proof}
By Facts 2, 3 and 4, if $f\prec g\prec h$ then $\left(f\circ g\right)\circ h$
and $f\circ\left(g\circ h\right)$ are defined, and by Fact 1 this
implies that $\left(f\circ g\right)\circ h=f\circ\left(g\circ h\right)$.
\end{proof}
Poloids are semigroupoids with effective left and right units. Such
units can be added to transformation semigroupoids in a quite natural
way.
\begin{defn}
\label{def8}A\emph{ transformation poloid} $\mathscr{P}_{\!X}$ is
a\emph{ }transformation semigroupoid $\mathscr{S}_{\!X}$ such that
if $f\in\mathscr{S}_{\!X}$ then $I\!d_{\mathrm{dom}\left(f\right)},I\!d_{\mathrm{cod}\left(f\right)}\in\mathscr{S}_{\!X}$. 
\end{defn}
\begin{fact}
\label{f5}Let $\mathscr{P}_{X}$ be a transformation poloid. For
any $f\in\mathscr{P}_{X}$, $I\!d_{\mathrm{dom}\left(f\right)}$ and
$I\!d_{\mathrm{cod}\left(f\right)}$ are units.
\end{fact}
\begin{proof}
If $f,g\in\mathscr{P}_{X}$ and $I\!d_{\mathrm{dom}\left(f\right)}\circ g$
is defined then 
\begin{gather*}
\mathrm{dom}\!\left(I\!d_{\mathrm{dom}\left(f\right)}\circ g\right)=\mathrm{dom}\!\left(g\right),\\
\mathrm{cod}\!\left(I\!d_{\mathrm{dom}\left(f\right)}\circ g\right)=\mathrm{cod}\!\left(I\!d_{\mathrm{dom}\left(f\right)}\right)=\mathrm{dom}\!\left(I\!d_{\mathrm{dom}\left(f\right)}\right)=\mathrm{cod}\!\left(g\right),
\end{gather*}
 and $I\!d{}_{\mathrm{dom}\left(\mathsf{f}\right)}\!\left(g\!\left(x\right)\right)=g\!\left(x\right)$
for all $x\in\mathrm{dom}\!\left(g\right)$. Hence, $I\!d_{\mathrm{dom}\left(f\right)}\circ g=g$.

Also, if $f,h\in\mathscr{P}_{X}$ and $h\circ I\!d_{\mathrm{dom}\left(f\right)}$
is defined then
\begin{gather*}
\mathrm{dom}\!\left(h\right)=\mathrm{cod}\!\left(I\!d_{\mathrm{dom}\left(f\right)}\right)=\mathrm{dom}\!\left(I\!d_{\mathrm{dom}\left(f\right)}\right)=\mathrm{dom}\!\left(h\circ I\!d_{\mathrm{dom}\left(f\right)}\right),\\
\mathrm{cod}\!\left(h\right)=\mathrm{cod}\!\left(h\circ I\!d_{\mathrm{dom}\left(f\right)}\right),
\end{gather*}
 and $h\!\left(I\!d{}_{\mathrm{dom}\left(\mathsf{f}\right)}\!\left(x\right)\right)=h\!\left(x\right)$
for all $x\in\mathrm{dom}\!\left(I\!d_{\mathrm{dom}\left(f\right)}\right)=\mathrm{dom}\!\left(h\right)$,
so $h\circ I\!d_{\mathrm{dom}\left(f\right)}=h$.

We have thus shown that $I\!d{}_{\mathrm{dom}\left(f\right)}$ is
a unit. 

It is shown similarly that if $I\!d{}_{\mathrm{cod}\left(f\right)}\circ g$
is defined then $I\!d{}_{\mathrm{cod}\left(f\right)}\circ g=g$, and
if $h\circ I\!d_{\mathrm{cod}\left(f\right)}$ is defined then $h\circ I\!d_{\mathrm{cod}\left(f\right)}=h$,
so $I\!d_{\mathrm{cod}\left(f\right)}$ is a unit as well.
\end{proof}
\begin{fact}
\label{f6}Let $\mathscr{P}_{X}$ be a transformation poloid. For
any $f\in\mathscr{P}_{X}$, $f\circ I\!d_{\mathrm{dom}\left(f\right)}$
and $I\!d_{\mathrm{cod}\left(f\right)}\circ f$ are defined. 
\end{fact}
\begin{proof}
We have $\mathrm{dom}\!\left(f\right)=\mathrm{dom}\!\left(I\!d_{\mathrm{dom}\left(f\right)}\right)=\mathrm{cod}\!\left(I\!d_{\mathrm{dom}\left(f\right)}\right)$
and $\mathrm{dom}\!\left(I\!d_{\mathrm{cod}\left(f\right)}\right)=\mathrm{cod}\!\left(f\right)$.
\end{proof}
\begin{thm}
\label{the2}A transformation poloid is a poloid.
\end{thm}
\begin{proof}
Immediate from Facts 5 and 6.
\end{proof}
\begin{rem}
{\small{}We have considered two requirements for $\mathsf{f}\circ$$\mathsf{g}$
or $f\circ g$ being defined, namely that $\mathrm{dom}\!\left(\mathsf{f}\right)\supseteq\mathsf{im}\!\left(\mathsf{g}\right)$
or that $\mathrm{dom}\!\left(f\right)=\mathsf{\mathrm{cod}}\!\left(g\right)$.
Other definitions are common in the literature. Instead of requiring
that $\mathrm{dom}\!\left(\mathsf{f}\right)\supseteq\mathsf{im}\!\left(\mathsf{g}\right)$,
it is often required that $\mathrm{dom}\!\left(\mathsf{f}\right)\cap\mathsf{im}\!\left(\mathsf{g}\right)\neq\emptyset$,
and instead of requiring that $\mathrm{dom}\!\left(f\right)=\mathsf{\mathrm{cod}}\!\left(g\right)$,
it is sometimes required that $\mathrm{dom}\!\left(f\right)=\mathsf{\mathrm{im}}\!\left(g\right)$.
Of these alternative definitions, the first one tends to be too weak
for present purposes, while the second one tends to be too restrictive.}{\small \par}

{\small{}For example, if we stipulate that $\mathsf{f}\circ\mathsf{g}$
is defined if and only if $\mathrm{dom}\!\left(\mathsf{f}\right)\cap\mathsf{im}\!\left(\mathsf{g}\right)\neq\emptyset$
then $\mathsf{f}\circ\mathsf{g}$ and $\mathsf{g}\circ\mathsf{h}$
being defined does not imply that $\left(\mathsf{f}\circ\mathsf{g}\right)\circ\mathsf{h}$
and $\mathsf{f}\circ\left(\mathsf{g}\circ\mathsf{h}\right)$ are defined,
contrary to Fact \ref{f2}.}{\small \par}

{\small{}Also, if we stipulate that $f\circ g$ is defined if and
only if $\mathrm{dom}\!\left(f\right)=\mathsf{\mathrm{im}}\!\left(g\right)$
and let $f,g$ be total transformations on $X$, then $f\circ g$
is defined only if $g$ is surjective so that $\mathrm{im}\!\left(g\right)=X$.
Thus, $\overline{\mathcal{F}}{}_{\!X}=\left\{ f\mid f:X\rightarrow X\right\} $
is not a monoid under this function composition. As monoids are poloids,
this anomaly suggests that the condition $\mathrm{dom}\!\left(f\right)=\mathsf{\mathrm{im}}\!\left(g\right)$
for $f\circ g$ to be defined is not appropriate in the context of
poloids.}{\small \par}

{\small{}On the other hand, stipulating that $f\circ g$ is defined
if and only if $\mathrm{dom}\!\left(f\right)\supseteq\mathsf{im}\!\left(g\right)$
does not give a fully associative binary operation (Example 2). This
is a fatal flaw for many purposes, including representing poloids
as magmas of transformations.}{\small \par}

{\small{}We note that the exact formalization of the notion of ``partial
function'' is important. A ``partial function'' $\mathfrak{f}$
is often defined as being equipped only with a domain and an image
(range), and then there are only three reasonable ways of composing
``partial transformations'': $\mathfrak{f}\circ\mathfrak{g}$ is
defined if and only if $\mathrm{dom}\!\left(\mathfrak{f}\right)\cap\mathsf{im}\!\left(\mathsf{\mathfrak{g}}\right)\neq\emptyset$
or $\mathrm{dom}\!\left(\mathfrak{f}\right)\supseteq\mathsf{im}\!\left(\mathfrak{g}\right)$
or $\mathrm{dom}\!\left(\mathfrak{f}\right)=\mathsf{\mathrm{im}}\!\left(\mathsf{\mathfrak{g}}\right)$.
But according to Definition \ref{def1}, ``partial functions'' have
codomains of their own, so we can stipulate that $\mathfrak{f}\circ\mathfrak{g}$
is defined if and only if $\mathrm{dom}\!\left(\mathfrak{f}\right)=\mathsf{\mathrm{cod}}\!\left(\mathfrak{g}\right)$,
and this turns out to be just what we need when specializing magmas
of ``partial transformations'' to semigroupoids and poloids.}{\small \par}
\end{rem}

\section{Poloids and categories}

\subsection{Elementary properties of abstract poloids}

Recall that a poloid $P$ is a  magma satisfying the following conditions:
\begin{lyxlist}{1}
\item [{(P1).}] For any $x\prec y\prec z\in P$, $\left(xy\right)\!z$
and $x\!\left(yz\right)$ are defined and $\left(xy\right)\!z=x\!\left(yz\right)$.
\item [{(P2).}] \noindent For any $x\!\in\!P$ there are units $\epsilon_{\!x},\varepsilon_{\!x}\!\in\!P$
such that $\epsilon_{\!x}x$ and $x\varepsilon_{\!x}$ are defined.
\end{lyxlist}
Let us derive some elementary properties of poloids as abstract algebraic
structures.\newpage{}
\begin{prop}
\label{pro1}Let $P$ be a poloid and $e\in P$ a unit. Then $ee$
is defined and $ee=e$.
\end{prop}
\begin{proof}
Let $\epsilon_{e}\in P$ be an effective left unit for the unit $e$.
Then, $\epsilon_{e}e$ is defined and $e=\epsilon_{e}e=\epsilon_{e}$,
implying the assertion.
\end{proof}
By Proposition \ref{pro1}, every unit is an effective left and right
unit for itself.
\begin{prop}
\label{pro2}Let $P$ be a poloid. If $\epsilon_{x}$ and $\epsilon_{x}'$
are effective left units for $x\in P$ then $\epsilon_{x}=\epsilon_{x}'$,
and if $\varepsilon{}_{x}$ and $\varepsilon_{x}'$ are effective
right units for $x\in P$ then $\varepsilon_{x}=\varepsilon_{x}'$.
\end{prop}
\begin{proof}
By assumption, $\epsilon_{x}x$ and and $\epsilon_{x}'x$ are defined
and equal to $x$, so $\epsilon_{x}\!\left(\epsilon_{x}'x\right)$
is defined. Thus, $\left(\epsilon_{x}\epsilon_{x}'\right)\!x$ is
defined, so $\epsilon_{x}\epsilon_{x}'$ is defined. As $\epsilon_{x}$
and $\epsilon_{x}'$ are units, this implies $\epsilon_{x}=\epsilon_{x}\epsilon_{x}'=\epsilon_{x}'$.
The uniqueness of the effective right unit for $x$ is proved in the
same way.
\end{proof}
Note that if $xy$ is defined then $\left(\epsilon_{x}x\right)\!y$
is defined so $\epsilon_{x}\!\left(xy\right)$ is defined and $\epsilon_{x}\!\left(xy\right)=\left(\epsilon_{x}x\right)\!y=xy=\epsilon_{xy}\!\left(xy\right)$,
so by Proposition \ref{pro2} we have $\epsilon_{x}=\epsilon_{xy}$.
A similar argument shows that $\varepsilon_{y}=\varepsilon_{xy}$. 

Also note that in a groupoid, where $xx^{-1}$ and $x^{-1}x$ are
defined and units, we have $x\left(x^{-1}x\right)=x$, $x^{-1}\left(xx^{-1}\right)=x^{-1}$,
$\left(xx^{-1}\right)x=x$, and $\left(x^{-1}x\right)x^{-1}=x^{-1}$,
where the four left-hand sides are defined. Thus, by Proposition \ref{pro2}
we have $xx^{-1}=\epsilon_{x}=\varepsilon_{x^{-1}}$ and $x^{-1}x=\varepsilon_{x}=\epsilon_{x^{-1}}$.
\begin{prop}
\label{pro3}Every poloid $P$ can be equipped with surjective functions
\begin{gather*}
s:P\rightarrow E,\qquad x\mapsto\epsilon_{x},\\
t:P\rightarrow E,\qquad x\mapsto\varepsilon_{x},
\end{gather*}
 where $E$ is the set of all units in $P$ and $s\!\left(e\right)=t\!\left(e\right)=e$
for all $e\in E$.
\end{prop}
\begin{proof}
Immediate from (P2), Proposition \ref{pro1} and Proposition \ref{pro2}.
\end{proof}
\begin{prop}
\label{pro4}Let $P$ be a poloid. For any $x,y\in P$, $xy$ is defined
if and only if $\varepsilon_{x}=\epsilon_{y}$.
\end{prop}
\begin{proof}
If $xy$ is defined then $\left(x\varepsilon_{x}\right)\!y$ is defined,
so $\varepsilon_{x}y$ is defined and as $\varepsilon_{x}$ is a unit
we have $\varepsilon_{x}y=y=\epsilon_{y}y$, so $\varepsilon_{x}=\epsilon_{y}$
by Proposition \ref{pro2}. Conversely, if $\varepsilon_{x}=\epsilon_{y}$
then $\varepsilon_{x}y$ is defined, and as $x\varepsilon_{x}$ is
defined, $\left(x\varepsilon_{x}\right)\!y=xy$ is defined. 
\end{proof}
A \emph{total poloid} is a poloid $P$ whose binary operation $\uppi$
is a total function.
\begin{prop}
A total poloid has only one unit.
\end{prop}
\begin{proof}
For any pair $e,e'\in P$ of units, $ee'$ is defined so $e=ee'=e'$.
\end{proof}
\begin{prop}
\label{pro5}A poloid with only one unit is a monoid.
\end{prop}
\begin{proof}
Let $P$ be a poloid. By assumption, there is a unique unit $e\in P$
such that $e=\varepsilon_{x}=\epsilon_{y}$ for any $x,y\in P$. Therefore,
it follows from Proposition \ref{pro4} that $xy$ and $yz$ are defined
for any $x,y,z\in P$. Hence, $\left(xy\right)\!z$ and $x\!\left(yz\right)$
are defined and equal for any $x,y,z\in P$. Also, $x=\epsilon_{x}x=x\varepsilon_{x}$
for any $x\in P$ implies $x=ex=xe$ for any $x\in P$.
\end{proof}
A poloid can thus be regarded as a generalized monoid, and also as
a generalized groupoid; in fact, poloids generalize groups via monoids
and via groupoids.
\begin{prop}
\label{pro6}A groupoid with only one unit is a group.
\end{prop}
\begin{proof}
A monoid with inverses is a group.
\end{proof}

\subsection{Subpoloids, poloid homomorphisms and poloid actions}

Recall that a submonoid of a monoid $M$ is a monoid $M'$ such that
$M'$ is a submagma of $M$ and the unit in $M'$ is the unit in $M$.
Subpoloids can be defined similarly.
\begin{defn}
\label{def9}A \emph{subpoloid} of a poloid $P$ is a poloid $P'$
such that $P'$ is a submagma of $P$ and every unit in $P'$ is a
unit in $P$.
\end{defn}
Homomorphisms and actions of poloids similarly generalize homomorphisms
and actions of monoids.
\begin{defn}
\label{def10}Let $P$ and $Q$ be poloids. A \emph{poloid homomorphism}
from $P$ to $Q$ is a total function $\phi:P\rightarrow Q$ such
that
\begin{lyxlist}{1}
\item [{(1)}] if $x,y\in P$ and $xy$ is defined then $\phi\!\left(x\right)\!\phi\!\left(y\right)$
is defined and $\phi\!\left(xy\right)=\phi\!\left(x\right)\!\phi\!\left(y\right)$;
\item [{(2)}] if $e$ is a unit in $P$ then $\phi\!\left(e\right)$ is
a unit in $Q$.
\end{lyxlist}
A \emph{poloid isomorphism} is a poloid homomorphism $\phi$ such
that the inverse function $\phi^{-1}$ exists and is a poloid homomorphism.
\end{defn}
Note that $\phi\!\left(x\right)=\phi\!\left(\epsilon_{x}x\right)=\phi\!\left(\epsilon_{x}\right)\!\phi\!\left(x\right)$
by (1) and $\phi\!\left(\epsilon_{x}\right)$ is a unit by (2) in
Definition \ref{def10}, so by Proposition \ref{pro2} we have $\phi\!\left(\epsilon_{x}\right)=\epsilon_{\phi\left(x\right)}$.
Dually, $\phi\!\left(\varepsilon_{x}\right)=\varepsilon_{\phi\left(x\right)}$.

Let $P$ be a poloid, let $Q$ be a magma and assume that there exists
a total function $\phi:P\rightarrow Q$ satisfying (1) and (2) in
Definition \ref{def10} and also such that (1') if $\phi\!\left(x\right)\!\phi\!\left(y\right)$
is defined then $xy$ is defined. It is easy to verify that then $\phi\!\left(P\right)$
is a magma, (P1) is satisfied in $\phi\!\left(P\right)$, and if $x'=\phi\!\left(x\right)\in\phi\!\left(P\right)$
then $\phi\!\left(\epsilon_{x}\right)$ ($\phi\!\left(\varepsilon_{x}\right))$
is an effective left (right) unit for $x'$, so $\phi\!\left(P\right)$
is a poloid.
\begin{defn}
\label{def11}A \emph{poloid action} of a poloid $P$ on a set $X$
is a total function 
\[
\alpha:P\rightarrow\alpha\!\left(P\right)\subseteq\overline{\mathscr{F}}_{\!X}
\]
which is a poloid homomorphism such that if $e\in P$ is a unit then
$\alpha\!\left(e\right)\in\alpha\!\left(P\right)$ is an identity
transformation $I\!d_{\mathrm{dom}\left(\alpha\left(e\right)\right)}$
on $X$. 

A \emph{prefunction poloid action} of a poloid $P$ on $X$ is similarly
a total function 
\[
\upalpha:P\rightarrow\upalpha\!\left(P\right)\subseteq\overline{\mathscr{R}}_{\!X}
\]
which is a poloid homomorphism such that if $e\in P$ is a unit then
$\upalpha\!\left(e\right)\in\upalpha\!\left(P\right)$ is an identity
pretransformation $\mathsf{Id}{}_{\mathrm{dom}\left(\upalpha\left(e\right)\right)}$
on $X$. 
\end{defn}
A poloid action $\alpha$ thus assigns to each $x\in P$ a non-empty
transformation
\[
\alpha\!\left(x\right):X\nrightarrow X,\qquad t\mapsto\alpha\!\left(x\right)\!\left(t\right)
\]
such that if $xy$ is defined then $\alpha\!\left(x\right)\circ\alpha\!\left(y\right)$
is defined and $\alpha\!\left(xy\right)=\alpha\!\left(x\right)\circ\alpha\!\left(y\right)$,
and for each unit $e$ in $P$ its image $\alpha\!\left(e\right)$
is a unit in $\alpha\!\left(P\right)$ such that $\alpha\!\left(e\right)\!\left(t\right)=t$
for each $t\in\mathrm{dom}\!\left(\alpha\!\left(e\right)\right)$.
\begin{rem}
{\small{}The definition of a poloid homomorphism given here implies
the usual definition of a monoid homomorphism. The definition of a
monoid action obtained from Definition \ref{def11} is also the usual
one. Specifically, a monoid action $\alpha$ of $M$ on a set $X$
is a function
\[
\alpha:M\rightarrow\alpha\!\left(M\right)\subseteq\overline{\mathcal{F}}_{\!X}
\]
such that $\alpha\!\left(xy\right)\!\left(t\right)=\alpha\!\left(x\right)\circ\alpha\!\left(y\right)\!\left(t\right)$
and $\alpha\!\left(e\right)\!\left(t\right)=t$ for all $x,y\in M$
and all $t\in X$. Denoting $\alpha\!\left(x\right)\!\left(t\right)$
by $x\cdot t$, this is rendered as $\left(xy\right)\cdot t=x\cdot\left(y\cdot t\right)$
and $e\cdot t=t$. Note that $\alpha\!\left(e\right)=I\!d_{X}$ is
a unit in $\overline{\mathcal{F}}_{\!X}$ and thus in $\alpha\!\left(M\right)$.}{\small \par}
\end{rem}

\subsection{Poloids as transformation poloids}

Recall that every transformation poloid is, indeed, a poloid. Up to
isomorphism, there are, in fact, no other poloids.
\begin{lem}
\label{lem2}For any poloid $P$, there is a prefunction poloid action
\[
\upmu:P\rightarrow\upmu\!\left(P\right)\subseteq\mathscr{\overline{R}}\!_{P},\qquad x\mapsto\upmu\!\left(x\right)
\]
 of $P$ on $P$ such that $\upmu$ is a poloid isomorphism.
\end{lem}
\begin{proof}
Set $\upmu\!\left(x\right)=\left(\overline{\upmu\!\left(x\right)},\mathrm{dom}\!\left(\upmu\!\left(x\right)\right)\right)$,
where $\overline{\upmu\!\left(x\right)}\!\left(t\right)=xt$ for all
$t\in\mathrm{dom}\!\left(\upmu\!\left(x\right)\right)$ and $\mathrm{dom}\!\left(\upmu\!\left(x\right)\right)=\left\{ t\!\mid\!xt\;\mathrm{defined}\right\} $.
Then $\upmu\!\left(x\right)$ is a prefunction $P\nRightarrow P$,
and $\upmu\!\left(x\right)$ is non-empty for each $x\in P$ since
$x\varepsilon_{x}$ is defined for each $x\in P$.

Furthermore, $\overline{\upmu\!\left(x\right)}\!\left(\varepsilon{}_{x}\right)=x\varepsilon{}_{x}=x$
for any $x\in P$, and also $\overline{\upmu\!\left(y\right)}\!\left(\varepsilon{}_{x}\right)=y\varepsilon{}_{x}=y$
for any $y\in P$ such that $y\varepsilon{}_{x}$ is defined since
$\varepsilon{}_{x}$ is a unit. Hence, if $x\neq y$ and $y\varepsilon{}_{x}$
is defined then $\overline{\upmu\!\left(x\right)}\!\left(\varepsilon{}_{x}\right)\neq\overline{\upmu\!\left(y\right)}\!\left(\varepsilon{}_{x}\right)$,
so $\overline{\upmu\!\left(x\right)}\neq\overline{\upmu\!\left(y\right)}$;
if $x\neq y$ and $y\varepsilon{}_{x}$ is not defined then $\mathrm{dom}\left(\alpha\!\left(x\right)\right)\neq\mathrm{dom}\left(\alpha\!\left(y\right)\right)$
since $x\varepsilon{}_{x}$ is defined. Thus, $\upmu$ is a bijection. 

For any fixed $x,y\in P$ such that $xy$ is defined, $\left(xy\right)\!t$
is defined if and only if $t\in P$ is such that $yt$ is defined.
Thus, $\mathrm{im}\!\left(\upmu\!\left(y\right)\right)=\left\{ yt\mid yt\:\mathrm{defined}\right\} =\left\{ yt\mid x\!\left(yt\right)\;\mathrm{defined}\right\} \subseteq\left\{ t\mid xt\;\mathrm{defined}\right\} =$
$\mathrm{dom}\!\left(\upmu\!\left(x\right)\right)$ and $\left\{ t\mid\left(xy\right)\!t\:\mathrm{defined}\right\} =\left\{ t\mid yt\:\mathrm{defined}\right\} $,
so if $xy$ is defined then $\upmu\!\left(x\right)\circ\upmu\!\left(y\right)$
is defined and $\mathrm{dom}\!\left(\upmu\!\left(xy\right)\right)=\mathrm{dom}\!\left(\upmu\!\left(y\right)\right)=\mathrm{dom}\!\left(\upmu\!\left(x\right)\circ\upmu\!\left(y\right)\right)$.
Also, if $xy$ is defined and $t\in\mathrm{dom}\!\left(\upmu\!\left(xy\right)\right)=\mathrm{dom}\!\left(\upmu\!\left(y\right)\right)$,
meaning that $yt$ is defined, then $\left(xy\right)\!t$ and $x\!\left(yt\right)$
are defined and equal, and as $\left(xy\right)\!t=\overline{\upmu\!\left(xy\right)}\!\left(t\right)$
for all $t\in\mathrm{dom}\!\left(\upmu\!\left(xy\right)\right)$ and
$x\!\left(yt\right)=\overline{\upmu\!\left(x\right)}\circ\overline{\upmu\!\left(y\right)}\left(t\right)$
for all $t\in\mathrm{dom}\!\left(\upmu\!\left(x\right)\circ\upmu\!\left(y\right)\right)$,
this implies that if $xy$ is defined then $\upmu\!\left(xy\right)=\upmu\!\left(x\right)\circ\upmu\!\left(y\right)$.

Conversely, if $\upmu\!\left(x\right)\circ\upmu\!\left(y\right)$
is defined so that $\left\{ t\mid yt\;\mathrm{defined}\right\} =\mathrm{dom}\!\left(\upmu\!\left(y\right)\right)=\mathrm{dom}\!\left(\upmu\!\left(x\right)\circ\upmu\!\left(y\right)\right)=\left\{ t\mid x\!\left(yt\right)\;\mathrm{defined}\right\} $,
then $yt$ defined implies $x\!\left(yt\right)$ defined for any fixed
$x,y\in P$. But this implication does not hold if $xy$ is not defined;
then, $x\!\left(y\varepsilon_{y}\right)$ is not defined although
$y\varepsilon_{y}$ is defined. Hence, if $\upmu\!\left(x\right)\circ\upmu\!\left(y\right)$
is defined then $xy$ must be defined. Therefore, $\upmu\!\left(x\right)\circ\upmu\!\left(y\right)=\upmu\!\left(xy\right)\in\upmu\!\left(P\right)$,
so $\upmu\!\left(P\right)$ is a magma with $\circ$ as binary operation.

Let $e\in P$ be a unit. If $\upmu\!\left(e\right)\circ\upmu\!\left(x\right)$
is defined then $ex$ is defined so $\upmu\!\left(x\right)=\upmu\!\left(ex\right)=\upmu\!\left(e\right)\circ\upmu\!\left(x\right)$,
and if $\upmu\!\left(x\right)\circ\upmu\!\left(e\right)$ is defined
then $xe$ is defined so $\upmu\!\left(x\right)=\upmu\!\left(xe\right)=\upmu\!\left(x\right)\circ\upmu\!\left(e\right)$.
Thus, $\upmu\!\left(e\right)\in\upmu\!\left(P\right)$ is a unit.

Conversely, if $f'=\upmu\!\left(f\right)\in\upmu\!\left(P\right)$
is a unit and $fx$ is defined then $\upmu\!\left(f\right)\circ\upmu\!\left(x\right)$
is defined and $\upmu\!\left(fx\right)=\upmu\!\left(f\right)\circ\upmu\!\left(x\right)=\upmu\!\left(x\right)$,
so $fx=x$ since $\upmu$ is injective. Similarly, if $\upmu\!\left(f\right)\in\upmu\!\left(P\right)$
is a unit and $xf$ is defined then $xf=x$. Hence, $f\in P$ is a
unit. 

Thus, we have shown that $\upmu$ satisfies the conditions labeled
(1), (1') and (2) in Section 3.2, so $\upmu\!\left(P\right)$ is a
poloid. Also, (1) and (2) in Definition \ref{def10} are satisfied
by both $\upmu$ and $\upmu^{-1}$, so $\upmu:P\rightarrow\upmu\!\left(P\right)$
is a poloid isomorphism.

The observation that if $e\in P$ is a unit then $\upmu\!\left(e\right)\!\left(t\right)=et=t$
for all $t\in\mathrm{dom}\!\left(\upmu\!\left(e\right)\right)$, so
that $\upmu\!\left(e\right)$ is an identity pretransformation $\mathsf{Id}{}_{\mathrm{dom}\left(\upmu\left(e\right)\right)}$,
completes the proof.
\end{proof}
\begin{lem}
\label{lem3}For any poloid $P$ and function $\upmu$ defined as
in Lemma \ref{lem2}, there is a total function $\tau:\upmu\!\left(P\right)\rightarrow\tau\!\left(\upmu\!\left(P\right)\right)\subseteq\mathscr{\overline{F}}\!_{P}$
such that 
\begin{enumerate}
\item $\tau$ is bijective;
\item $\upmu\!\left(x\right)\circ\upmu\!\left(y\right)$ is defined if and
only if $\tau\!\left(\upmu\!\left(x\right)\right)\circ\tau\!\left(\upmu\!\left(y\right)\right)$
is defined;
\item if $\upmu\!\left(x\right)\circ\upmu\!\left(y\right)$ is defined then
$\tau\!\left(\upmu\!\left(x\right)\circ\upmu\!\left(y\right)\right)=\tau\!\left(\upmu\!\left(x\right)\right)\circ\tau\!\left(\upmu\!\left(y\right)\right)$;
\item if $e\in P$ is a unit then $\tau\!\left(\upmu\!\left(e\right)\right)\in\mathscr{\overline{F}}\!_{P}$
is a unit and identity transformation $I\!d{}_{\mathrm{dom}\left(\tau\left(\upmu\left(e\right)\right)\right)}$.
\end{enumerate}
\end{lem}
\begin{proof}
For any prefunction $\upmu\!\left(x\right)=\left(\overline{\upmu\!\left(x\right)},\mathrm{dom}\!\left(\upmu\!\left(x\right)\right)\right):P\nRightarrow P$,
$x\in P$, the tuple 
\[
\left/\upmu\!\left(x\right)\right/=\left(\overline{\upmu\!\left(x\right)},\mathrm{dom}\!\left(\upmu\!\left(x\right)\right),\mathrm{dom}\!\left(\upmu\!\left(\epsilon_{x}\right)\right)\right)
\]
is a function $P\nrightarrow P$ for which $\overline{\left/\upmu\!\left(x\right)\right/}=\overline{\upmu\!\left(x\right)}$,
$\mathrm{dom}\!\left(\left/\upmu\!\left(x\right)\right/\right)=\mathrm{dom}\!\left(\upmu\!\left(x\right)\right)$
and $\mathrm{cod}\!\left(\left/\upmu\!\left(x\right)\right/\right)=\mathrm{dom}\!\left(\upmu\!\left(\epsilon_{x}\right)\right)$.
In fact, $\epsilon_{x}x$ is defined, so $\upmu\!\left(\epsilon_{x}\right)\circ\upmu\!\left(x\right)$
is defined, so $\mathrm{cod}\!\left(\left/\upmu\!\left(x\right)\right/\right)\!=\mathrm{dom}\!\left(\upmu\!\left(\epsilon_{x}\right)\right)\supseteq\mathrm{im}\!\left(\upmu\!\left(x\right)\right)=\overline{\upmu\!\left(x\right)}\!\left(\mathrm{dom}\!\left(\upmu\!\left(x\right)\right)\right)=\overline{\left/\upmu\!\left(x\right)\right/}\!\left(\mathrm{dom}\!\left(\left/\upmu\!\left(x\right)\right/\right)\right)$\linebreak{}
$=\mathrm{im}\!\left(\left/\upmu\!\left(x\right)\right/\right)$,
as required. Thus, there is a total function
\[
\tau:\mathscr{\overline{R}}\!_{P}\supseteq\upmu\!\left(P\right)\rightarrow\tau\!\left(\upmu\!\left(P\right)\right)\subseteq\mathscr{\overline{F}}\!_{P},\qquad\upmu\!\left(x\right)\mapsto\left/\upmu\!\left(x\right)\right/.
\]

It remains to prove (1) \textendash{} (4). (1) and (2) are obvious.
Also, $\mathrm{dom}\!\left(\left/\upmu\!\left(x\right)\circ\upmu\!\left(y\right)\right/\right)=\mathrm{dom}\!\left(\upmu\!\left(x\right)\circ\upmu\!\left(y\right)\right)=\mathrm{dom}\!\left(\upmu\!\left(y\right)\right)=\mathrm{dom}\!\left(\left/\upmu\!\left(y\right)\right/\right)=\mathrm{dom}\!\left(\left/\upmu\!\left(x\right)\right/\circ\left/\upmu\!\left(y\right)\right/\right)$
and\linebreak{}
 $\mathrm{cod}\!\left(\left/\upmu\!\left(x\right)\circ\upmu\!\left(y\right)\right/\right)=\mathrm{cod}\!\left(\left/\upmu\!\left(xy\right)\right/\right)=\mathrm{dom}\!\left(\upmu\!\left(\epsilon_{xy}\right)\right)=\mathrm{dom}\!\left(\upmu\!\left(\epsilon_{x}\right)\right)=\mathrm{cod}\!\left(\left/\upmu\!\left(x\right)\right/\right)=\mathrm{cod}\!\left(\left/\upmu\!\left(x\right)\right/\circ\left/\upmu\!\left(y\right)\right/\right),$
so $\left/\upmu\!\left(x\right)\circ\upmu\!\left(y\right)\right/=\left/\upmu\!\left(x\right)\right/\circ\left/\upmu\!\left(y\right)\right/$.
Concerning (4), $\upmu\!\left(e\right)$ is a unit and identity pretransformation
$\mathsf{Id}{}_{\mathrm{dom}\left(\upmu\left(e\right)\right)}$ in
$\mathscr{\overline{R}}\!_{P}$, so it suffices to note that $\mathrm{cod}\!\left(\left/\upmu\!\left(e\right)\right/\right)=\mathrm{dom}\!\left(\upmu\!\left(\epsilon_{e}\right)\right)=\mathrm{dom}\!\left(\upmu\!\left(e\right)\right)=\mathrm{dom}\!\left(\left/\upmu\!\left(e\right)\right/\right)$.
\end{proof}
\begin{thm}
\label{lem2-1}For any poloid $P$, there is a poloid action
\[
\alpha:P\rightarrow\alpha\!\left(P\right)\subseteq\mathscr{\overline{F}}\!_{P},\qquad x\mapsto\alpha\!\left(x\right)
\]
 of $P$ on $P$ such that $\alpha$ is a poloid isomorphism and $\alpha\!\left(P\right)$
equipped with $\circ$ is a transformation poloid.
\end{thm}
\begin{proof}
First set $\alpha=\tau\circ\upmu$ and use Lemmas \ref{lem2} and
\ref{lem3} to prove the first part of the theorem. It remains to
show that $\alpha\!\left(P\right)$ is a transformation poloid. Recall
that $\upmu$ and $\tau$ are injective so that $\alpha$ is injective,
and note that $\mathrm{dom}\!\left(\alpha\!\left(x\right)\right)=\mathrm{dom}\!\left(\alpha\!\left(x\varepsilon_{x}\right)\right)=\mathrm{dom}\!\left(\alpha\!\left(x\right)\circ\alpha\!\left(\varepsilon_{x}\right)\right)=\mathrm{dom}\!\left(\alpha\!\left(\varepsilon_{x}\right)\right)$
and that identity transformations, such as $\alpha\!\left(\varepsilon_{x}\right)$
and $\alpha\!\left(\epsilon_{y}\right)$, are determined by their
domains. Hence, we have
\begin{align*}
 & \mathrm{dom}\!\left(\alpha\!\left(x\right)\right)=\mathrm{cod}\!\left(\alpha\!\left(y\right)\right)\\
\Longleftrightarrow\quad & \mathrm{dom}\!\left(\alpha\!\left(\varepsilon_{x}\right)\right)=\mathrm{dom}\!\left(\alpha\!\left(\epsilon_{y}\right)\right)\\
\Longleftrightarrow\quad & \alpha\!\left(\varepsilon_{x}\right)=\alpha\!\left(\epsilon_{y}\right)\\
\Longleftrightarrow\quad & \varepsilon_{x}=\epsilon_{y}\\
\Longleftrightarrow\quad & xy\;\mathrm{defined}\\
\Longleftrightarrow\quad & \alpha\!\left(x\right)\circ\alpha\!\left(y\right)\;\mathrm{defined}.
\end{align*}
Thus the poloid of transformations $\alpha\!\left(P\right)$ is a
transformation semigroupoid by \linebreak{}
Definition \ref{def7}. Also, if $\alpha\!\left(x\right)\in\alpha\!\left(P\right)$
then $\alpha\!\left(\epsilon_{x}\right),\alpha\!\left(\varepsilon_{x}\right)\in\alpha\!\left(P\right)$,
$\alpha\!\left(\epsilon_{x}\right)=I\!d_{\mathrm{dom}\left(\alpha\left(\epsilon_{x}\right)\right)}=I\!d{}_{\mathrm{cod}\left(\alpha\left(x\right)\right)}$
and $\alpha\!\left(\varepsilon{}_{x}\right)=I\!d{}_{\mathrm{dom}\left(\alpha\left(\varepsilon{}_{x}\right)\right)}=I\!d{}_{\mathrm{dom}\left(\alpha\left(x\right)\right)}$,
so the transformation semigroupoid $\alpha\!\left(P\right)$ is a
transformation poloid by Definition \ref{def8}.
\end{proof}
\begin{cor}
\label{the3-1}Any poloid is isomorphic to a transformation poloid.
\end{cor}
This is a 'Cayley theorem' for poloids; it generalizes similar isomorphism
theorems for groupoids, monoids and groups. Note, though, that $\alpha\!\left(P\right)$
is not only a \emph{poloid of transformations} isomorphic to $P$,
but actually a \emph{transformation poloid} isomorphic to $P$, so
Corollary \ref{the3-1} is stronger than a straight-forward generalization
of the 'Cayley theorem' as usually stated.

\subsection{Categories as poloids}

It is no secret that a poloid is the same as a small arrows-only category.
In various guises, (P1), (P2) and Propositions \ref{pro1} \textendash{}
\ref{pro4} appear as axioms or theorems in category theory. The two-axiom
system proposed here is related to the set of ``Gruppoid'' axioms
given by Brandt \cite{key-1}, and essentially equivalent to axiom
systems used by Freyd \cite{key-2}, Hastings \cite{key-6}, and others.
By Proposition 3, one can define functions $s:x\mapsto\epsilon_{x}$
and $t:x\mapsto\varepsilon_{x}$; axiom systems using these two functions
but equivalent to the one given here, as used by Freyd and Scedrov
\cite{key-3}, currently often serve to define arrows-only categories.

Concepts from category theory can be translated into the the language
of poloids and vice versa. For example, an initial object in a category
corresponds to some unit $\epsilon\in P$ such that for every unit
$e\in P$ there is a unique $x\in P$ such that $\epsilon x$ and
$xe$ are defined (hence, $\epsilon x=x=xe)$. More significantly,
in the language of category theory a subpoloid is a subcategory, and
a poloid homomorphism is a functor.

Looking at categories as ``webs of monoids'' does lead to some shift
of emphasis and perspective, however. In particular, whereas the notion
of a category acting on a set is not emphasized in texts on category
theory, the corresponding notion of a poloid action is central when
regarding categories as poloids. For example, recall that letting
a group act on itself we obtain Cayley's theorem for groups. Similarly,
letting a poloid act on itself we have obtained a Cayley theorem for
poloids \cite{key-7}, corresponding to Yoneda's lemma for categories.
Poloid actions are also a tool that can be used to define ordinary
(small) two-sorted categories in terms of poloids \textendash{} we
let a poloid $P$ act on a set $O$ in a special way, then interpreting
the elements of $P$ as morphisms and the elements of $O$ acted on
by $P$ as objects. 

Applying an algebraic perspective on category theory may thus lead
to more than merely a reformulation of category theory, especially
as the algebraic structures related to categories are also linked
to specific magmas of transformations.

\bigskip{}

\appendix

\section{Constellations}

A\emph{ constellation} \cite{key-4,key-7}, is defined in \cite{key-5}
as follows:
\begin{quote}
A {[}\emph{left}{]}\emph{ constellation} is a structure $P$ of signature
$(\cdot,D)$ consisting of a class $P$ with a partial binary operation
and unary operation $D$ {[}...{]} that maps onto the set of \emph{projections}
$E\subseteq P$, so that $E=\left\{ D(x)\mid x\in P\right\} $, and
such that for all $e\in E$, $ee$ exists and equals $e$, and for
which, for all $x,y,z\in P$: 

\begin{lyxlist}{00.00.0000}
\item [{(C1)}] if $x\cdot(y\cdot z)$ exists then so does $(x\cdot y)\cdot z$,
and then the two are equal; 
\item [{(C2)}] $x\cdot(y\cdot z)$ exists if and only if $x\cdot y$ and
$y\cdot z$ exist; 
\item [{(C3)}] for each $x\in P$, $D(x)$ is the unique left identity
of $x$ in $E$ (i.e. it satisfies $D(x)\cdot x=x$); 
\item [{(C4)}] for $a\in P$ and $g\in E$, if $a\cdot g$ exists then
it equals $a$.
\end{lyxlist}
\end{quote}
It turns out that constellations generalize poloids. Recall that by
Definition \ref{def3} a semigroupoid is a partial magma such that
if (a) $x\!\left(yz\right)$ is defined or (b) $\left(xy\right)\!z$
is defined or (c) $xy$ and $yz$ are defined then $x\!\left(yz\right)$
and $\left(xy\right)\!z$ are defined and $x\!\left(yz\right)=\left(xy\right)\!z$.
Removing (a), we obtain the following definition.
\begin{defn}
\label{def12}A \emph{right-directed semigroupoid} is a magma $P$
such that, for any $x,y,z\in P$, if $\left(xy\right)\!z$ is defined
or $xy$ and $yz$ are defined then $\left(xy\right)\!z$ and $x\!\left(yz\right)$
are defined and $x\!\left(yz\right)=\left(xy\right)\!z$. 
\end{defn}
The condition in this definition corresponds to conditions (C1) and
(C2) in \cite{key-5} except for some non-substantial differences.
First, we are defining here the left-right dual of the notion defined
by (C1) and (C2). This amounts to a difference in notation only, deriving
from the fact that functions are composed from left to right in \cite{key-5}
while they are composed from right to left here. Second, it is not
necessary to postulate that if $\left(xy\right)\!z$ is defined then
$xy$ and $yz$ are defined, in accordance with (C2), because, by
Definition \ref{def12}, if $\left(xy\right)\!z$ is defined then
$x(yz)$ is defined, so $xy$ and $yz$ are defined. Finally, in \cite{key-5}
$P$ is assumed to be a class rather than a set; this difference has
to do with set-theoretic considerations that need not concern us here. 

We shall need some generalizations of the unit concept. First, a \emph{left
unit} in $P$ is an element $\epsilon$ of $P$ such that $\epsilon x=x$
for all $x\in P$ such that $\epsilon x$ is defined, while a \emph{right
unit} in $P$ is an element $\varepsilon$ of $P$ such that $x\varepsilon=x$
for all $x\in P$ such that $x\varepsilon$ is defined. Also, a \emph{local
left unit} $\lambda_{x}$ for $x\in P$ is an element of $P$ such
that $\lambda_{x}x$ is defined and $\lambda_{x}x=x$, while a \emph{local
right unit} $\rho_{x}$ for $x\in P$ is an element of $P$ such that
$x\rho_{x}$ is defined and $x\rho_{x}=x$.
\begin{defn}
\label{def13}A \emph{right poloid} is a right-directed semigroupoid
$P$ such that for any $x\in P$ there is a unique left unit $\varphi_{x}\in P$
such that $\varphi_{x}$ is a local right unit for $x$.
\end{defn}
\begin{prop}
\label{pro7}Let $P$ be a right poloid. If $\epsilon\in P$ is a
left unit then $\epsilon\epsilon$ is defined and $\epsilon\epsilon=\epsilon$.
\end{prop}
\begin{proof}
Let $\varphi_{\epsilon}\in P$ be a local right unit for the left
unit $\epsilon$. Then $\epsilon\varphi_{\epsilon}$ is defined and
$\varphi_{\epsilon}=\epsilon\varphi_{\epsilon}=\epsilon$, and this
implies the assertion.
\end{proof}
Thus, the left unit $\epsilon$ is the unique local right unit $\varphi_{e}$
for itself.

Disregarding (C1) and (C2), which were incorporated in Definition
\ref{def12}, the requirements stated in the definition cited above
can be summed up as follows:
\begin{lyxlist}{00.00.0000}
\item [{(C)}] For each $x\in P$, there is exactly one $D\!\left(x\right)\in E=\left\{ D\!\left(x\right)\mid x\in P\right\} $
such that $D\!\left(x\right)\cdot x$ is defined and $D\!\left(x\right)\cdot x=x$,
and every $e\in E$ is a right unit in $P$ and such that $e\cdot e$
is defined and equal to $e$. 
\end{lyxlist}
Using (C), it can be proved as in Proposition \ref{pro7} that if
$f\in P$ is a right unit then $f\cdot f$ is defined and $f\cdot f=f$,
so $f$ is the unique local left unit $D\!\left(f\right)$ for itself.
Thus, $E$ equals the set of right units in $P$, since conversely
every $e\in E$ is a right unit in $P$ by (C). As all right units
are idempotent, this means that the requirement that all elements
of $E$ are idempotent is redundant, so (C) can be simplified to:
\begin{lyxlist}{00.00.0000}
\item [{(C{*})}] For each $x\in P$, there is exactly one right unit $D\!\left(x\right)\in P$
such that $D\!\left(x\right)\cdot x$ is defined and $D\!\left(x\right)\cdot x=x$. 
\end{lyxlist}
In our terminology, this means, of course, that for any $x\in P$
there is a unique left unit $\varphi_{x}$ in $P$ such $\varphi_{x}$
is a local right unit for $x$. We conclude that a (small) constellation
is just a right poloid; note that $D$ is just the function $x\mapsto\varphi_{x}$.
Proposition \ref{pro7} generalizes Proposition \ref{pro1}, and there
are also natural generalizations of Propositions \ref{pro2} \textendash{}
\ref{pro4} to right poloids.

It should be pointed out that in \cite{key-5} an alternative definition
of constellations is also given; this definition is essentially the
same as Definition \ref{def13} here (see Proposition 2.9 in \cite{key-5}).
So while the definition of constellations cited above reflects the
historical development of that notion, it has been shown here and
in \cite{key-5} that a more direct approach can also be used.\medskip{}

Let us also look at the transformation systems corresponding to constellations.
\newpage{}
\begin{thm}
\label{the3}A pretransformation magma is a right-directed semigroupoid.
\end{thm}
\begin{proof}
Use Facts \ref{f1} \textendash{} \ref{f3} in Section 2.3.
\end{proof}
A \emph{domain} pretransformation magma is a pretransformation magma
$\mathscr{R}_{\!X}$ such that if $\mathsf{f}\!\in\!\mathscr{R}_{\!X}$
then $\mathsf{Id}_{\mathrm{dom}\left(\mathsf{f}\right)}\!\in\!\mathscr{R}_{\!X}$.
Corresponding to Theorem \ref{the2} in Section 2.3, we have the following
result.
\begin{thm}
\label{the4}A domain pretransformation magma is a right poloid.
\end{thm}
\begin{proof}
In view of Theorem \ref{the3}, it suffices to show that for any $\mathsf{f}\in\mathscr{R}_{\!X}$
there is a unique left unit $\upvarphi_{\mathsf{f}}\in\mathscr{R}_{\!X}$
such that $\mathsf{f}\circ\upvarphi_{\mathsf{f}}$ is defined and
equal to $\mathsf{f}$, namely $\mathsf{Id}_{\mathrm{dom}\left(\mathsf{f}\right)}$.

If $\mathsf{f},\mathsf{g}\in\mathscr{R}_{\!X}$ and $\mathsf{Id}_{\mathrm{dom}\left(\mathsf{f}\right)}\circ\mathsf{g}$
is defined so that $\mathrm{dom}\!\left(\mathsf{Id}_{\mathrm{dom}\left(\mathfrak{\mathsf{f}}\right)}\right)\supseteq\mathrm{im}\!\left(\mathsf{g}\right)$
then 
\begin{gather*}
\mathrm{dom}\!\left(\mathsf{Id}_{\mathrm{dom}\left(\mathsf{f}\right)}\circ\mathsf{g}\right)=\mathrm{dom}\!\left(\mathfrak{\mathsf{g}}\right),\\
\mathsf{Id}_{\mathrm{dom}\left(\mathsf{f}\right)}\circ\mathsf{g}\left(x\right)=\mathsf{Id}_{\mathrm{dom}\left(\mathsf{f}\right)}\!\left(\mathsf{g}\!\left(x\right)\right)=\mathsf{g}\!\left(x\right)
\end{gather*}
for all $x\!\in\!\mathrm{dom}\!\left(\mathsf{g}\right)$, meaning
that $\mathfrak{\mathsf{Id}}_{\mathrm{dom}\left(\mathsf{f}\right)}\!\circ\mathsf{g}=\mathsf{g}$.
Thus, $\mathfrak{\mathsf{Id}}_{\mathrm{dom}\left(\mathsf{f}\right)}$
is a left unit in $\mathscr{R}_{\!X}$. 

Also, $\mathrm{dom}\!\left(\mathsf{f}\right)=\mathrm{dom}\!\left(\mathsf{Id}_{\mathrm{dom}\left(\mathsf{f}\right)}\right)=\mathrm{im}\!\left(\mathsf{Id}_{\mathrm{dom}\left(\mathsf{f}\right)}\right)$,
so $\mathsf{f}\circ\mathsf{Id}_{\mathrm{dom}\left(\mathfrak{\mathsf{f}}\right)}$
is defined, and
\begin{gather*}
\mathrm{dom}\!\left(\mathfrak{\mathsf{f}}\circ\mathsf{Id}_{\mathrm{dom}\left(\mathsf{f}\right)}\right)=\mathrm{dom}\!\left(\mathsf{Id}_{\mathrm{dom}\left(\mathfrak{\mathsf{f}}\right)}\right)=\mathrm{dom}\!\left(\mathfrak{\mathsf{f}}\right),\\
\mathfrak{\mathsf{f}}\circ\mathsf{Id}_{\mathrm{dom}\left(\mathsf{f}\right)}\left(x\right)=\mathsf{f}\!\left(\mathsf{Id}{}_{\mathrm{dom}\left(\mathsf{f}\right)}\!\left(x\right)\right)=\mathfrak{\mathsf{f}}\!\left(x\right)
\end{gather*}
for all $x\in\mathrm{dom}\!\left(\mathsf{Id}_{\mathrm{dom}\left(\mathsf{f}\right)}\right)=\mathrm{dom}\!\left(\mathfrak{\mathsf{f}}\right)$.
Thus, $\mathsf{f}\circ\mathsf{Id}_{\mathrm{dom}\left(\mathsf{f}\right)}$
is defined and equal to $\mathsf{f}$, so $\mathsf{Id}{}_{\mathrm{dom}\left(\mathsf{f}\right)}\in\mathscr{R}_{\!X}$
is a left unit $\upvarphi_{\mathsf{f}}$ such that $\mathsf{f}\circ\upvarphi_{\mathsf{f}}$
is defined and equal to $\mathsf{f}$.

It remains to show that $\mathsf{Id}_{\mathrm{dom}\left(\mathsf{f}\right)}$
is the only such $\upvarphi_{\mathsf{f}}$. Let $\upepsilon\in\mathscr{R}_{\!X}$
be a left unit. Then $\mathfrak{\mathsf{Id}}_{\mathrm{dom}\left(\upepsilon\right)}\in\mathscr{R}_{\!X}$
and as $\mathrm{dom}\!\left(\upepsilon\right)=\mathrm{dom}\!\left(\mathfrak{\mathsf{Id}}_{\mathrm{dom}\left(\upepsilon\right)}\right)=\mathrm{im}\!\left(\mathfrak{\mathsf{Id}}_{\mathrm{dom}\left(\upepsilon\right)}\right)$,
so that $\upepsilon\circ\mathfrak{\mathsf{Id}}_{\mathrm{dom}\left(\upepsilon\right)}$
is defined, we have $\upepsilon\circ\mathfrak{\mathsf{Id}}_{\mathrm{dom}\left(\upepsilon\right)}=\mathfrak{\mathsf{Id}}_{\mathrm{dom}\left(\upepsilon\right)}$.
On the other hand,
\begin{gather*}
\mathrm{dom}\left(\upepsilon\circ\mathsf{Id}_{\mathrm{dom}\left(\upepsilon\right)}\right)=\mathrm{dom}\left(\mathsf{Id}_{\mathrm{dom}\left(\upepsilon\right)}\right)=\mathrm{dom}\!\left(\upepsilon\right),\\
\upepsilon\circ\mathsf{Id}_{\mathrm{dom}\left(\upepsilon\right)}\left(x\right)=\upepsilon\!\left(\mathsf{Id}_{\mathrm{dom}\left(\upepsilon\right)}\!\left(x\right)\right)=\upepsilon\!\left(x\right)
\end{gather*}
for all $x\in\mathrm{dom}\!\left(\mathsf{Id}_{\mathrm{dom}\left(\upepsilon\right)}\right)=\mathrm{dom}\!\left(\upepsilon\right)$,
so $\upepsilon\circ\mathsf{Id}_{\mathrm{dom}\left(\upepsilon\right)}=\upepsilon$.
Thus $\upepsilon=\mathsf{Id}_{\mathrm{dom}\left(\upepsilon\right)}$,
so $\upvarphi_{\mathsf{f}}=\mathsf{Id}_{\mathrm{dom}\left(\upvarphi_{\mathsf{f}}\right)}=\mathsf{Id}_{\mathrm{dom}\left(\mathfrak{\mathsf{f}}\circ\upvarphi_{\mathsf{f}}\right)}=\mathsf{Id}_{\mathrm{dom}\left(\mathsf{f}\right)}$. 
\end{proof}
With Theorem \ref{the2} and Corollary \ref{the3-1}  in mind, one
might expect, given Theorem \ref{the4}, that conversely every right
poloid is isomorphic to some domain pretransformation magma (regarded
as a right poloid). Indeed, any poloid can be embedded in a pretransformation
magma by Lemma \ref{lem2}, and it can be shown that $\upmu\!\left(\varepsilon_{x}\right)=\mathsf{Id}_{\mathrm{dom}\left(\upmu\left(x\right)\right)}$,
so any poloid can actually be embedded in a domain pretransformation
magma. Also, the proof of Lemma \ref{lem2} uses almost only properties
of poloids that they share with right poloids. There is one crucial
exception, though: both $\varepsilon_{x}$ and $\varphi_{x}$ are
local right units, but in addition $\varepsilon_{x}$ is a unit while
$\varphi_{x}$ is just a left unit. The fact that $\varepsilon_{x}$
is a unit is used to prove that $x\mapsto\upmu\!\left(x\right)$ is
injective, and this is not true for all right poloids.
\begin{example}
\label{ex3} The magma defined by the Cayley table below is a right
poloid with $x=\varphi_{x}$ and $y=\varphi_{y}$, but $\upmu\!\left(x\right)=\upmu\!\left(y\right)$.

\[
\begin{array}{ccc}
 & x & y\\
x & x & y\\
y & x & y
\end{array}
\]
\end{example}
This suggests that we look for an additional condition on right poloids
to ensure that $x\mapsto\upmu\!\left(x\right)$ is injective. On finding
such a condition, we can prove a weakened converse of Theorem \ref{the4}
by an argument similar to the proof of Lemma \ref{lem2}.

Adapting a definition in \cite{key-5}, we say that a right poloid
such that if $\varphi_{x}\varphi_{y}$ and $\varphi_{y}\varphi_{x}$
are defined then $\varphi_{x}=\varphi_{y}$ is \emph{normal}. (The
poloid in Example \ref{ex3} is not normal.) This notion is the key
to the following three results:
\begin{thm}
\label{the6}A domain pretransformation magma is a normal right poloid.
\end{thm}
\begin{proof}
In a domain pretransformation magma, $\upvarphi_{\mathsf{f}}=\mathsf{Id}_{\mathrm{dom}\left(\mathsf{f}\right)}$.
Thus, $\upvarphi_{\mathsf{f}}\circ\upvarphi_{\mathsf{g}}$ is defined
if and only if $\mathrm{dom}\!\left(\mathsf{Id}_{\mathrm{dom}\left(\mathsf{f}\right)}\right)\supseteq\mathrm{im}\!\left(\mathsf{Id}_{\mathrm{dom}\left(\mathsf{g}\right)}\right)=\mathrm{dom}\!\left(\mathsf{Id}_{\mathrm{dom}\left(\mathsf{g}\right)}\right)$,
or equivalently $\mathrm{dom}\!\left(\mathsf{f}\right)\supseteq\mathrm{dom}\!\left(\mathsf{g}\right)$,
so if $\upvarphi_{\mathsf{f}}\circ\upvarphi_{\mathsf{g}}$ and $\upvarphi_{\mathsf{g}}\circ\upvarphi_{\mathsf{f}}$
are defined then $\mathrm{dom}\!\left(\mathsf{f}\right)=\mathrm{dom}\!\left(\mathsf{g}\right)$,
so $\mathsf{Id}_{\mathrm{dom}\left(\mathsf{f}\right)}=\mathsf{Id}_{\mathrm{dom}\left(\mathsf{g}\right)}$
or equivalently $\upvarphi_{\mathsf{f}}=\upvarphi_{\mathsf{g}}$.
\end{proof}
\begin{lem}
\label{lem4} In a normal right poloid, the correspondence $x\mapsto\upmu\!\left(x\right)$
is injective.
\end{lem}
\begin{proof}
Assume that $x\neq y$. If $\mathrm{dom}\!\left(\upmu\!\left(x\right)\right)\neq\mathrm{dom}\!\left(\upmu\!\left(y\right)\right)$
then $\upmu\!\left(x\right)\neq\upmu\!\left(y\right)$ as required.
Otherwise, $\mathrm{dom}\!\left(\upmu\!\left(x\right)\right)=\mathrm{dom}\!\left(\upmu\!\left(y\right)\right)$,
and as $x\varphi_{x}$ and $y\varphi_{y}$ are defined we have $\varphi_{x},\varphi_{y}\in\mathrm{dom}\!\left(\upmu\!\left(x\right)\right)=\mathrm{dom}\!\left(\upmu\!\left(y\right)\right)$.
Thus, $x\varphi_{y}$ is defined, so $\left(x\varphi_{x}\right)\!\varphi_{y}$
is defined, so $x\!\left(\varphi_{x}\varphi_{y}\right)$ is defined,
so $\varphi_{x}\varphi_{y}$ is defined. Similarly, $y\varphi_{x}$
is defined, so $\varphi_{y}\varphi_{x}$ is defined. Therefore, $\varphi_{x}=\varphi_{y}$,
so $\upmu\!\left(x\right)\!\left(\varphi_{x}\right)=x$ and $\upmu\!\left(y\right)\!\left(\varphi_{x}\right)=\upmu\!\left(y\right)\!\left(\varphi_{y}\right)=y$,
so again $\upmu\!\left(x\right)\neq\upmu\!\left(y\right)$.
\end{proof}
Using Lemma \ref{lem4} and proceeding as in the proof of Lemma \ref{lem2},
keeping in mind that $\varphi_{\upmu\left(x\right)}=\upmu\!\left(\varepsilon_{x}\right)=\mathsf{Id}_{\mathrm{dom}\left(\upmu\left(x\right)\right)}$,
we obtain the following result:
\begin{thm}
\label{the7}A normal right poloid can be embedded in a domain pretransformation
magma.
\end{thm}
Theorems \ref{the6} and \ref{the7} correspond to Proposition 2.23
in \cite{key-5}.

Let us look at another way of narrowing down the notion of a right
poloid so that any right poloid considered can be embedded in a domain
pretransformation magma. Consider the relation $\leq$ on a right
poloid $P$ given by $x\leq y$ if and only if $y\varphi_{x}$ is
defined and $x=y\varphi_{x}$. The relation $\leq$ is obviously reflexive,
and if $x\leq y$ and $y\leq z$ then (a) $y\varphi_{x}=\left(z\varphi_{y}\right)\!\varphi_{x}$
is defined so that $z\!\left(\varphi_{y}\varphi_{x}\right)=z\varphi_{x}$
is defined and (b) $x=y\varphi_{x}=\left(z\varphi_{y}\right)\!\varphi_{x}=z\!\left(\varphi_{y}\varphi_{x}\right)=z\varphi_{x}$,
so $\leq$ is transitive as well. Hence, $\leq$ is a preorder, called
the \emph{natural preorder} on $P$, so $\leq$ is a partial order
if and only if it is antisymmetric. \emph{A} right poloid such that
$\epsilon\leq\epsilon'$ and $\epsilon'\leq\epsilon$ implies $\epsilon=\epsilon'$
for any left units $\epsilon,\epsilon'\in P$ is said to be \emph{unit-posetal}.

Recall that for any left unit $\epsilon\in P$ we have $\varphi_{\epsilon}=\epsilon$,
so $\varphi_{\varphi_{x}}=\varphi_{x}$. Thus, $\varphi_{x}\leq\varphi_{y}$
if and only $\varphi_{y}\varphi_{x}$ is defined and $\varphi_{x}=\varphi_{y}\varphi_{x}$,
so as $\varphi_{y}$ is a left unit we have $\varphi_{x}\leq\varphi_{y}$
if and only $\varphi_{y}\varphi_{x}$ is defined. Hence, we obtain
the following results.
\begin{thm}
\label{the8}A right poloid is unit-posetal if and only if it is normal.
\end{thm}
\begin{thm}
\label{the9}A domain pretransformation magma is a unit-posetal right
poloid.
\end{thm}
\begin{thm}
\label{the10}A unit-posetal right poloid can be embedded in a domain
pretrans\-formation magma.
\end{thm}
If we specialize the concept of a unit-posetal right poloid by adding
more requirements, the analogue of Theorem \ref{the9} need of course
not hold. In particular, the partial order on the left units is not
necessarily a semilattice.
\begin{example}
Set $X=\left\{ 1,2,3\right\} $ and $\mathscr{R}_{\!X}=\left\{ \mathsf{Id}{}_{\left\{ 1,2\right\} },\mathsf{Id}{}_{\left\{ 2,3\right\} }\right\} $
with $\mathsf{f}\circ\mathsf{g}$ defined as usual when $\mathrm{dom}\left(\mathsf{f}\right)\supseteq\mathrm{im}\left(\mathsf{g}\right)$.
Then $\mathscr{R}_{\!X}$ is a domain pretransformation magma where
$\mathsf{Id}{}_{\left\{ 1,2\right\} }\leq\mathsf{Id}{}_{\left\{ 1,2\right\} }$
and $\mathsf{Id}{}_{\left\{ 2,3\right\} }\leq\mathsf{Id}{}_{\left\{ 2,3\right\} }$,
but this partial order is not a semilattice.
\end{example}
More broadly, let $\boldsymbol{A}$ denote a class of abstract algebraic
structures corresponding to a class $\boldsymbol{C}$ of concrete
magmas of correspondences (functions, prefunctions etc.) in the sense
that any $\boldsymbol{c}$ in $\boldsymbol{C}$ belongs to $\boldsymbol{A}$
when certain operations in $\boldsymbol{C}$ are interpreted as the
operations in $\boldsymbol{A}$. Note that this does not imply that
any $\boldsymbol{a}$ in $\boldsymbol{A}$ can be embedded in some
$\boldsymbol{c}$ in $\boldsymbol{C}$. In particular, if $\boldsymbol{A}$
is a class of generalized groups, with axioms merely defining a generalized
group operation and (optional) generalized identities and inverses,
then the fact that the axioms defining $\boldsymbol{A}$ are satisfied
for any concrete magma $\boldsymbol{c}$ in $\boldsymbol{C}$ does
not provide a strong reason to expect that any $\boldsymbol{a}$ satisfying
these axioms can be embedded in some $\boldsymbol{c}$ in $\boldsymbol{C}$.
As we have just seen, the relation between right poloids and domain
pretransformation magmas is asymmetrical in this respect. (One-sided)
restriction semigroups \cite{key-4} provide another example of this
phenomenon. 
\begin{example}
Let $\boldsymbol{A}$ be the class of semigroups such that for each
$\boldsymbol{a}$ in $\boldsymbol{A}$ and each $x\in\boldsymbol{a}$
there is a unique local right unit for $x$ in $\boldsymbol{a}$.
Let $\boldsymbol{C}$ be the class of semigroups of functional relations
on a given set where the binary operation is composition of relations
and such that for each $\boldsymbol{c}$ in $\boldsymbol{C}$ and
each $\mathtt{f}\in\boldsymbol{c}$ the functional relation $\mathtt{Id}{}_{\mathrm{dom}\left(\mathtt{f}\right)}$
belongs to $\boldsymbol{c}$. Then any $\boldsymbol{c}$ in $\boldsymbol{C}$
belongs to $\boldsymbol{A}$, with $\mathtt{Id}{}_{\mathrm{dom}\left(\mathtt{f}\right)}$
the local right unit for $\mathtt{f}$, but it is not the case that
any $\boldsymbol{a}$ in $\boldsymbol{A}$ can be embedded in some
$\boldsymbol{c}$ in $\boldsymbol{C}$. To ensure embeddability, $\boldsymbol{A}$
needs to be narrowed down by additional conditions, subject to the
restriction that $\boldsymbol{A}$ remains wide enough to accommodate
all $\boldsymbol{c}$ in $\boldsymbol{C}$.
\end{example}
We have seen that any transformation poloid is a poloid and that those
poloids which can be embedded in a transformation poloid are simply
all poloids, and a similar elementary symmetry exists for inverse
semigroups, but such cases are perhaps best regarded as ideal rather
than normal, reflecting the fact that poloids and inverse semigroups
are particularly natural algebraic structures.


\begin{thebibliography}{1}
\bibitem{key-1}Brandt H (1927), \textquotedbl{}Über eine Verallgemeinerung
des Gruppenbegriffes\textquotedbl{}, \emph{Mathematische Annalen},
96(1): 360\textendash 366.

\bibitem{key-2} Freyd P J (1964), \emph{Abelian Categories, an Introduction
to the Theory of Functors}. Harper \& Row.

\bibitem{key-3} Freyd P J, Scedrov A (1999), \emph{Categories, Allegories}.
North-Holland.

\bibitem{key-4}Gould V, Hollings C (2009), ``Restriction semigroups
and inductive constellations'', \emph{Communications in Algebra},
38(1): 261\textendash 287.

\bibitem{key-5}Gould V, Stokes T (2017), ``Constellations and their
relationship with categories'', \emph{Algebra Universalis}, 77(3):
271-304.

\bibitem{key-6}Higgins P J (1971), \emph{Categories and groupoids}.
Van Nostrand-Reinhold. 

\bibitem{key-7}Hollings C D (2007), \emph{Partial Actions of Semigroups
and Monoids}. PhD-thesis, York.

\bibitem{key-8}Ivan G (1996), ``Cayley theorem for monoidoids'',
\emph{Glasnik Matematicki}, 31(51): 73\textendash 82.

\bibitem{key-9}Wagner, V V (1953), ``The theory of generalized heaps
and generalized groups'', \emph{Mat. Sbornik }(\emph{N. S.}) 32(74):
545\textendash 632 (in Russian).
\end{thebibliography}
\end{document}